\documentclass{article}
\usepackage{amsthm}
\usepackage{fullpage}
\newtheorem{theorem}{Theorem}
\newtheorem{corollary}[theorem]{Corollary}

\newtheorem{conjecture}[theorem]{Conjecture}
\newtheorem{lemma}[theorem]{Lemma}
\newtheorem{clm}{Claim}
\newtheorem*{theoremA}{Theorem A}
\newtheorem*{theoremB}{Theorem B}
\newtheorem*{theorem5A}{Theorem 5(a)}
\newtheorem*{theorem5B}{Theorem 5(b)}
\def\mad{\textrm{mad}}
\def\C{C}

\begin{document}
\title{Injective colorings of sparse graphs}
\author{
Daniel W. Cranston\thanks{
\texttt{dcranston@vcu.edu},
Virginia Commonwealth University, Richmond, Virginia; DIMACS, Rutgers University, Piscataway, New Jersey.
} \and 
Seog-Jin Kim\thanks{
\texttt{skim12@konkuk.ac.kr},
Konkuk University, Seoul, Korea.  Research was supported by Basic Science Research Program
through the National Research Foundation of Korea(NRF) funded
by the Ministry of Education, Science and Technology(2009-0064177)} \and 
Gexin Yu\thanks{
\texttt{gyu@wm.edu}, 
College of William and Mary, Willliamsburg, VA 23185.
Research is partially supported by NSF grant DMS-0852452.
}}

\date{\today}
\maketitle

\begin{abstract}
Let $\mad(G)$ denote the maximum average degree (over all subgraphs) of $G$ and let $\chi_i(G)$ denote the injective chromatic number of $G$.  We prove that if $\mad(G) \leq \frac{5}{2}$, then $\chi_i(G)\leq\Delta(G) + 1$; and 
if $\mad(G) < \frac{42}{19}$, then $\chi_i(G)=\Delta(G)$.
Suppose that $G$ is a planar graph with girth $g(G)$ and $\Delta(G)\geq 4$.  
We prove that if $g(G)\geq 9$, then $\chi_i(G)\leq\Delta(G)+1$; similarly, 
if $g(G)\geq 13$, then $\chi_i(G)=\Delta(G)$.
\end{abstract}

\smallskip


\section{Introduction}
An \textit{injective coloring} of a graph $G$ is an assignment of colors to the vertices of $G$ so that any two vertices with a common neighbor receive distinct colors.  The \textit{injective chromatic number}, $\chi_i(G)$, is the minimum number of colors needed for an injective coloring.   Injective colorings have their origin in complexity theory~\cite{HKSS02}, and can be used in coding theory.  

Note that an injective coloring is not necessarily proper, and in fact,  $\chi_i(G) = \chi(G^{(2)})$, where the \textit{neighboring graph $G^{(2)}$} is defined by $V(G^{(2)})=V(G)$ and $E(G^{(2)})=\{uv:\mbox{$u$ and $v$ have a common}$ $\mbox{neighbor in $G$}\}$.  It is clear that $\Delta\leq\chi_i(G)\leq\Delta^2-\Delta+1$, where $\Delta$ is the maximum degree of graph $G$.  Graphs attaining the upper bound were characterized in~\cite{HKSS02}, and it was also shown that for every fixed $k\ge 3$ the problem of determining if a graph is injective $k$-colorable is NP-complete.

As one can see,  injective coloring is a close relative of the coloring of square of graphs and of $L(2,1)$-labeling, which have both been studied extensively.   Upper bounds on $\chi(G^2)$ and on the $L(2,1)$-labeling number $\lambda(G)$  are both upper bounds on $\chi_i(G)$.

Alon and Mohar~\cite{AM02} showed that if $G$ has girth at least $7$, then $\chi(G^2)$ could be as large as $\frac{c\Delta^2}{\log \Delta}$ (for some constant $c$), but not larger.  This gives an upper bound on $\chi_i(G)$ when graph $G$ has girth at least $7$.    The study of $\chi(G^2)$ has been largely focused on the the well-known Wenger's Conjecture~\cite{W77}.

\begin{conjecture}[Wenger~\cite{W77}]
If $G$ is a planar graph with maximum degree $\Delta$, then 
$\chi(G^2)\le 7$ when $\Delta=3$,  $\chi(G^2)\le \Delta+5$ when $4\le \Delta\le 7$, and $\chi(G^2)\le 3\Delta/2+1$ when $\Delta\ge 8$. 
\end{conjecture} 
\noindent
This conjecture and its variations have been studied extensively; see Borodin and Ivanova~\cite{BI09} or Dvo\v{r}\'ak, Kr\'al, Nejedl\'y, and \v{S}krekovski ~\cite{DKNS08} for a good survey.

Much effort has been spent on finding graphs with low injective chromatic numbers, namely graphs with $\chi_i(G)\le \Delta+c$, for some small constant $c$.   In Theorems~\ref{planar-girth} and~\ref{dhrtheorem}, we list some of the most recent related results. 

\begin{theorem}\label{planar-girth}
Let $G$ be a planar graph with maximum degree $\Delta(G)\ge D$ and girth $g(G)\ge g$.  Then

(a)  {\rm (Borodin, Ivanova, Neustroeva~\cite{BIN07})}  if 
$(D, g)\in \{(3, 24), (4, 15), (5, 13), (6, 12), (7, 11), (9, 10), (15, 8), (30, 7)\}$, then $\chi_i(G)\le \chi(G^2)=\Delta+1$.

(b)  {\rm (Borodin and Ivanova~\cite{BI09})} $\chi_i(G)\le \chi(G^2)\le \Delta+2$ if $(D, g)=(36, 6)$. 

(c)  {\rm (Luzar, Skrekovski, and Tancer~\cite{lst})} $\chi_i(G)\le \Delta+4$ if $g=5$; $\chi_i(G)\le \Delta+1$ if $g=10$; and $\chi_i(G)=\Delta$ if $g=19$.

(d) {\rm (Bu, Chen, Raspaud, and Wang~\cite{BCRW09})} $\chi_i(G)\le \Delta+2$ if $g=8$; $\chi_i(G)\le \Delta+1$ if $g=11$; and $\chi_i(G)=\Delta$ if $(D, g)\in \{(3, 20), (71, 7)\}$. 
\end{theorem}

Instead of studying planar graphs with high girth, some researchers consider graphs with  bounded {\it maximum average degree}, $\mad(G)$, where the average is taken over all subgraphs of $G$.  Note that every planar graph  $G$ with girth at least $g$ satisfies $\mad(G)<\frac{2g}{g-2}$.   Below are some results in terms of $\mad(G)$.

\begin{theorem}\label{dhrtheorem}
Let $G$ be a graph with maximum degree $\Delta(G)\ge D$ and $\mad(G)<m$. Then

(a) {\rm (Doyon, Hahn, and Raspaud~\cite{dhr})}  $\chi_i(G)\leq\Delta+8$ if $m=10/3$;  $\chi_i(G)\leq\Delta+4$ if $m=3$; and $\chi_i(G)\leq\Delta+3$ if $m=14/5$. 

\smallskip
(b) {\rm (Cranston, Kim, and Yu~\cite{CKY0000b})} $\chi_i(G)\le \Delta+2$ if $(D, m)=(4, 14/5)$; $\chi_i(G)\le 5$ if $(D, m)=(3, 36/13)$ and $m=36/13$ is sharp. 
\end{theorem}
\noindent

The major tool used in the proofs of Theorems~\ref{planar-girth} and~\ref{dhrtheorem} is the discharging method, which relies heavily on the idea of reducible subgraphs.  A {\it reducible subgraph} $H$ is a subgraph such that any coloring of $G-H$ can be extended to a coloring of $G$.   Since the coloring of $G-H$ will restrict the choice of colors on $H$,  these arguments work well when the graph $G$ is sparse.

In particular, if $\mad(G)$ is much smaller than $\Delta(G)$, then we
are guaranteed a vertex $v$ with degree much smaller than $\Delta(G)$.
 Such a vertex $v$ is a natural candidate to be included in our
reducible subgraph $H$, since $v$ has at least $\Delta(G)$ allowable
colors and only has a few restrictions on its color.  However, if
$\mad(G)$ is nearly as large as $\Delta(G)$, then we are not
guaranteed the presence of such a low degree vertex $v$.  Here it is
less clear how to proceed.
Thus, proving results when $\Delta(G)-\mad(G)$ is small is a much
harder task than proving analogous results when $\Delta(G)-\mad(G)$ is
larger.

In this paper,  we study graphs with low injective chromatic number, namely $G$ such that $\chi_i(G)\le \Delta+1$. 
We consider sparse graphs with bounded maximum average degree, which include planar graphs with high girth.  Our results below extend or generalize the corresponding results in Theorem~\ref{planar-girth}.   

\begin{theorem}\label{madthm}
Let $G$ be a graph with maximum degree $\Delta$.

(a) If $\mad(G) \leq \frac{5}{2}$, then $\chi_i(G)\leq\Delta + 1$.

(b) If $\mad(G) < \frac{42}{19}$, then $\chi_i(G)=\Delta$.
\end{theorem}

\begin{theorem}\label{girththm}
Let $G$ be a planar graph with girth $g(G)$ and maximum degree $\Delta\geq 4$.
  
(a) If $g(G)\geq 9$, then $\chi_i(G)\leq\Delta+1$.

(b) If $g(G)\geq 13$, then $\chi_i(G)=\Delta$.
\end{theorem}
\noindent

Like many similar results, we use discharging arguments in our proofs.   In most discharging arguments, the reducible subgraphs are of bounded size. Our contribution to injective coloring is using reducible configurations of arbitrary size, similar to the 2-alternating cycles introduced by Borodin~\cite{B89} and generalized by Borodin, Kostochka, and Woodall~\cite{BKW97}.  
\smallskip

Let $G$ be a Class 2 graph, that is, suppose the edge-chromatic number of $G$ is $\Delta+1$. Let $G'$ be the graph obtained from $G$ by inserting a degree 2 vertex on each edge.  Now $\chi_i(G')>\Delta$, for otherwise, we could color the edges of $G$ by the colors of the 2-vertices in $G'$ on the corresponding edges, which would give a $\Delta$-edge-coloring of $G$.   Here we list some Class 2 graphs; see Fiorini and Wilson~\cite{FW77} for more details. 

\begin{theorem}[Fiorini and Wilson~\cite{FW77}]
\label{fwtheorem}
(a) If $H$ is a regular graph with an even order, and G is a graph obtained from $H$ by inserting a new vertex into one edge of $H$, then $G$ is of Class 2. 

(b) For any integers $D\ge 3$ and $g\ge 3$, there is a Class 2 graph of maximum degree $D$ and girth $g$.
\end{theorem}

By combining Theorem~\ref{fwtheorem} with results on Class 2 graphs, we have the following corollary.

\begin{corollary}
There are planar graphs with $g(G)=8$ and $\chi_i(G)\ge \Delta+1$.   There are graphs with $\mad(G)=\frac{8}{3}$ and $\chi_i(G)\ge\Delta+1$.  For any $\Delta\ge 3$ and $g\ge3$,  there are graphs with maximum degree $\Delta$,  girth $2g$, and $\chi_i(G)\ge \Delta+1$. 
\end{corollary}

There are some gaps between the bounds in the above Corollary and Theorems~\ref{madthm} and ~\ref{girththm}.  Our bounds on mad and girth may be further improved if some clever idea is elaborated. 

\medskip

 When we extend a coloring of $G-H$ to a subgraph $H$, the colors available for vertices of $H$ are restricted, thus we will essentially supply a list of available colors for each vertex of $H$.  The following two classic theorems on list-coloring will be used heavily.  

\begin{theoremA}[Vizing~\cite{V76}]\label{vizinglemma}
Let $L$ be an assignment such that $|L(v)|\geq d(v)$ for all $v$ in a connected graph $G$.

(a) If $|L(y)| > d(y)$ for some vertex $y$, then $G$ is $L$-colorable.

(b) If $G$ is 2-connected and the lists are not all identical, then $G$ is $L$-colorable.
\end{theoremA}

We say that a graph is {\it degree-choosable} if it can be colored from lists when each vertex is given a list of size equal to its degree.

\begin{theoremB}[Borodin~\cite{B77},  Erd\H{o}s, Rubin, and Taylor~\cite{ERT79}] \label{west}
A graph $G$ fails to be degree-choosable if and only if every block is a complete graph or an odd cycle.
\end{theoremB}

Here we introduce some notation.   A $k$-vertex is a vertex of degree $k$; a $k^+$- and a $k^-$-vertex have degree at least and at most $k$, respectively.  A {\it thread} is a path with 2-vertices in its interior and $3^+$-vertices as its endpoints.  A $k$-thread has $k$ interior 2-vertices.  If $u$ and $v$ are the endpoints of a thread, then we say that $u$ and $v$ are {\it pseudo-adjacent}.  If a $3^+$-vertex $u$ is the endpoint of a thread containing a 2-vertex $v$, then we say that $v$ is a {\it nearby vertex} of $u$ and vice versa.   For other undefined notions,  we refer to \cite{W96}.

The paper is organized as follows.   In Section 2, we prove Theorem~\ref{madthm}(a). In Section 3, we prove Theorem~\ref{madthm}(b).  In Section 4, we prove Theorem~\ref{girththm}.

\section{Maximum average degree conditions that imply $\chi_i(G)\leq\Delta+1$}

We split the proof of Theorem~\ref{madthm}(a) into two lemmas.  We start with the case $\Delta\ge 4$, which only needs a simple discharging argument.  

\begin{lemma}
\label{52delta4}
Let $G$ be a graph with $\Delta\geq 4$.  If $\mad(G) \leq \frac{5}{2}$, then $\chi_i(G) \leq \Delta + 1$.
\end{lemma}

\begin{proof}
Let G be a minimal counterexample.  It is easy to see that G has no 1-vertex and G has no 2-thread.

Observe that G also must not contain the following subgraph: a 3-vertex $v$ adjacent to three 2-vertices such that one of these 2-vertices $u$ is adjacent to a second 3-vertex.  If $G$ contains such a subgraph, let $H$ be the set of $v$ and its three neighbors.  By minimality, we can injectively color $G-H$ with $\Delta+1$ colors; now we greedily color $H$, making sure to color $u$ last.

We now use a discharging argument, with initial charge $\mu(v) = d(v)$.
We have two discharging rules:
\medskip

(R1) Each 3-vertex divides a charge of $\frac12$ equally among its adjacent 2-vertices.

(R2) Each $4^+$-vertex sends a charge of $\frac13$ to each adjacent 2-vertex.
\medskip

\noindent
Now we show that $\mu^*(v)\geq \frac{5}{2}$ for each vertex v.

$d(v)=3$: $\mu^*(v) = 3 - \frac12 = \frac{5}{2}$

$d(v)\geq 4$: $\mu^*(v) \geq d(v) - \frac{d(v)}3 = \frac{2d(v)}3 \geq \frac83 > \frac{5}{2}$

$d(v) = 2$:
If $v$ is adjacent to a $4^+$-vertex, then $\mu^*(v) \geq 2 + \frac13 + \frac16 = \frac{5}{2}$.
If $v$ has two neighboring 3-vertices and each neighbor gives $v$ a charge of at least $\frac14$, then $\mu^*(v) \geq 2 + 2(\frac14) = \frac{5}{2}$.
However, if $v$ has two neighboring 3-vertices and at least one of them gives $v$ a charge of only $\frac16$, then $v$ is in a copy of the forbidden subgraph described above; so $\mu^*(v)\geq\frac{5}{2}$.

Hence, each vertex $v$ has charge $\mu^*(v) \geq \frac{5}{2}$.  Furthermore, each $4^+$-vertex has charge at least $\frac83$, so $\mad(G) > \frac{5}{2}$.
\end{proof}

The above argument fails when $\Delta=3$, since a 3-vertex $v$ may be adjacent to three 2-vertices such that one of these 2-vertices $u$ is adjacent to a second 3-vertex, which was forbidden when $\Delta\ge 4$.  We will still start with a minimal counterexample,  but  instead of finding reducible configurations in a local neighborhood, we will identify ones from global structure of the graphs.

\begin{lemma}\label{52delta3}
Let $G$ be a graph with $\Delta=3$.  If $\mad(G) \leq \frac{5}{2}$, then $\chi_i(G) \leq 4$.
\end{lemma}

\begin{proof}
We prove the more general statement: If $\Delta=3$ and $\mad(G)\leq \frac{5}{2}$, then $G$ can be injectively colored from {\it lists} of size 4.

Let $G$ be a minimal counterexample.  
It is easy to see that $G$ has no 1-vertex and $G$ has no 2-thread; in each case, we could delete the subgraph $H$, injectively color $G-H$, then greedily color $H$.
First, we consider the case $\mad(G) < \frac{5}{2}$. 
Let $G_{23}$ denote the subgraph induced by edges with one endpoint of degree 2 and the other of degree 3.  If each component of $G_{23}$ contains at most one cycle, then each component of $G_{23}$ contains at least as many 3-vertices as 2-vertices, so $\mad(G)\geq\frac{5}{2}$; this contradicts our assumption.  Hence, some component $H$ of $G$ contains a cycle $\C$ with a vertex $u$ on $C$ such that $d_H(u)=3$.  
Let $J = V(\C)\cup N(u)$.  Because $G$ is a minimal counterexample, we can injectively 4-color $G\setminus J$.  We will now use Theorem~A to extend the coloring of $G\setminus J$ to $J$.  

Since $\C$ is an even cycle, $G^{(2)}[J]$ consists of two components; one component contains $u$ and the other does not.  
We will use part (a) of Theorem~A to color the component that contains $u$ and we will use part $(b)$ to color the other component; note that this second component is 2-connected.  Let $L(v)$ be the list of colors available for each vertex $v$ before we color $G\setminus J$ and let $L'(v)$ be the list of colors available after we color $G\setminus J$.  To apply Theorem A as desired, we need to prove three facts: first, $|L'(v)|\geq d_{J^{(2)}}(v)$ for each $v\in J$; second, $|L'(u)| > d_{J^{(2)}}(u)$; and third, the lists $L'$ for the second component of $G^{(2)}[J]$ are not all identical.

First, observe that for each vertex $v\in J$ we have the inequality $|L(v)|=4\geq d_{G^{(2)}}(v)$;  
because each colored neighbor in $G^{(2)}$ of $v$ forbids only one color from use on $v$, this inequality implies that $|L'(v)|\geq d_{J^{(2)}}(v)$ for each $v\in J$.
Second, note that $|L(u)|=4 > 3 =d_{G^{(2)}}(u)$; this inequality implies that $|L'(u)| = 3 > 2 = d_{J^{(2)}}(u)$.
Third, let $x$ and $y$ be the two vertice on $\C$ that are adjacent to $u$ in $G$; note that $|L'(x)|=3$ and $|L'(y)|=3$, while $|L'(v)|=2$ for every other vertex $v$ in the second component of $J^{(2)}$.  Since vertices $x$ and $y$ have lists of size 3, while the other vertices have lists of size 2, it is clear that not all lists are identical.
 
Hence, we can color the first component of $J^{(2)}$ by Theorem~A part (a), and we can color the second component of $J^{(2)}$ by Theorem~A part (b).

Now consider the case $\mad(G) = \frac{5}{2}$.  If $G_{23}$ does not contain any cycle $\C$ with a vertex $u$ such that $d_{G_{23}}(u)=3$, then each component of $G_{23}$ is an even cycle.  This means that $d_{G^{(2)}}(v)=4$ for every vertex $v$.  By applying Theorem~B to $G^{(2)}$, we see that $\chi_i(G)\leq4$; it is straightforward to verify that in each component of $G^{(2)}$, some block is neither an odd cycle nor a clique.
\end{proof}

This completes the proof of Theorem~4(a).

\section{Maximum average degree conditions that imply $\chi_i(G)=\Delta$}

We split the proof of Theorem~\ref{madthm}(b) into two lemmas.  In Lemma~\ref{94delta4}, we prove a stronger result than we need for Theorem~\ref{madthm}(b), since our hypothesis here is $\mad(G)\leq\frac94$, rather than $\mad(G)\leq\frac{42}{19}$.

\begin{lemma}\label{94delta4}
Let $G$ be a graph with $\Delta \geq 4$.  If $\mad(G) \leq \frac94$, then $\chi_i(G) = \Delta$.
\end{lemma}

\begin{proof}
We always have $\chi_i(G)\geq\Delta$, so we only need to prove $\chi_i(G)\leq\Delta$.
Let $G$ be a minimal counterexample.
Clearly, $G$ has no 1-vertex and $G$ has no 4-thread.  Note that $G$ also has no 3-thread with a 3-vertex at one of its ends.  We use a discharging argument, with initial charge: $\mu(v) = d(v)$.  We have one discharging rule:

\medskip
(R1) Each $3^+$-vertex gives a charge of $\frac{1}{8}$ to each nearby 2-vertex.
\medskip

Now we show that $\mu^*(v)\geq \frac94$ for each vertex $v$.

2-vertex: $\mu^*(v)\geq 2 + 2(\frac18) = \frac94$.

3-vertex: $\mu^*(v)\geq 3 - 6(\frac18) = \frac94$.

$4^+$-vertex: $\mu^*(v)\geq d(v) - 3d(v)(\frac18) = \frac58d(v) \geq\frac{5}{2} > \frac94$.
\smallskip

Hence, each vertex $v$ has charge $\mu^*(v)\geq\frac94$.
Since each $4^+$-vertex $v$ has charge $\mu^*(v)\geq\frac{5}{2}$, we have $\mad(G)>\frac94$.
\end{proof}

When $\Delta=3$, our proof is more involved.  We consider a minimal counterexample $G$.   As above, $G$ has no 1-vertex and $G$ has no 4-thread.  

We form an auxiliary graph $H$ as follows.  Let $V(H)$ be the 3-vertices of $G$.
If $u$ and $v$ are ends of a 3-thread in $G$, then add the edge $uv$ to $H$.
Suppose instead that $u$ and $v$ are ends of a 2-thread in $G$.  If one of the other threads incident to $u$ is a 3-thread and the third thread incident to $u$ is either a 2-thread or 3-thread, then add edge $uv$ to $H$.

\begin{lemma}
If $\mad(G) < \frac{42}{19}$ and $\Delta(G)=3$, then $H$ contains a cycle with a vertex $v$ such that $d_H(v)=3$.
\label{clm1}
\end{lemma}

\begin{proof}
To prove Lemma~\ref{clm1}, it is sufficient to show that $H$ (or some subgraph of $H$) has average degree greater than 2. 

Form subgraph $\widehat{H}$ from $H$ by deleting all of the isolated vertices in $H$.
For $i\in\{0,1,\ldots,9\}$, let $a_i$ denote the number of 3-vertices of $G$ that have exactly $i$ nearby 2-vertices and let $\widehat{a}_i$ denote the number of 3-vertices of $G$ that have exactly $i$ nearby 2-vertices and also have a corresponding vertex in $\widehat{H}$.
Let $n$ and $\widehat{n}$ denote the number of vertices in $H$ and $\widehat{H}$, respectively.  Note that $\sum_{i=0}^9a_i = n$ and $\sum_{i=0}^9\widehat{a}_i = \widehat{n}$.  We now consider two weighted averages of the integers $0,1,\ldots, 9$; the first average uses the weights $a_i$ and the second average uses the weights $\widehat{a}_i$.

Let $V_2$ and $V_3$ denote the number of 2-vertices and 3-vertices in $G$.
Since $\mad(G) < \frac{42}{19}$, by simple algebra we deduce that $2V_2/V_3 > \frac{15}2$.  By rewriting this inequality in terms of the $a_i$s, we get: $\frac1n\sum_{i=0}^9a_ii > \frac{15}2$.  
Note that $8$ and $9$ are the only numbers in this weighted average that are larger than the average.
Thus, since $\widehat{a}_i=a_i$ for $i\in\{8,9\}$ and $\widehat{a}_i\leq a_i$ for $0\leq i\leq 7$, we also have $\frac1{\widehat{n}}\sum_{i=0}^9\widehat{a}_ii\geq\frac1n\sum_{i=0}^9a_ii > \frac{15}2$.  We need one more inequality, which we prove in the next paragraph.

The table below lists the values of three quantities: $i$, the minimum degree of a vertex in $\widehat{H}$ that has $i$ nearby 2-vertices in $G$, and the expression $2i/3-3$.  Note that for all values of $i$, we have $d_{\widehat{H}}(v)\geq 2i/3-3$.  We end the table at $i=4$, since thereafter $d_{\widehat{H}}(v)=1$ and $2i/3-3$ is negative.
\begin{center}
\begin{tabular}{c|c|c|c|c|c|c}
$i$ & 9 &  8  &  7  &  6  &  5  &  4  \\ \hline
$d_{\widehat{H}}(v)$ & 3 &  3  &  2  &  1  &  1  &  1 \\
$2i/3-3$ & 3 & 7/3 & 5/3 &  1  & 1/3 &-1/3 
\end{tabular}
\end{center}

By taking the average over all vertices in $\widehat{H}$ of the inequality $d_{\widehat{H}}(v)\geq2i/3-3$, we get the inequality 
\[
\frac1{\widehat{n}}\sum_{v\in V(\widehat{H})}d_{\widehat{H}}(v)\geq\frac1{\widehat{n}}\sum_{i=0}^9\widehat{a}_i\left(\frac23i-3\right).
\]
By expanding this second sum into a difference of two sums, then substituting the values given above for these two sums, we conclude that the average degree of $\widehat{H}$ is greater than 2.
This finishes the proof of Lemma~\ref{clm1}.
\end{proof}

Now we use Lemma~\ref{clm1} to show that $\chi_i(G) = 3$. 

\begin{lemma}
Let $G$ be a graph with $\Delta=3$.  If $\mad(G) < \frac{42}{19}$, then $\chi_i(G)= 3$.
\label{4219delta3}
\end{lemma}

\begin{proof}
 Let $\C$ be a cycle in $H$ that contains a vertex $v$ with $d_H(v)=3$.  Let $\C'$ be the cycle in $G$ that corresponds to $\C$; i.e. $\C'$ in $G$ passes through all the vertices corresponding to vertices of $\C$ in the same order that they appear on $\C$; furthermore, between each pair of successive 3-vertices on $\C'$ there are either two or three 2-vertices.  Let $J$ be the subgraph of $G$ consisting of cycle $\C'$ in $G$, together with the neighbor $w$ of $v$ that is not on $\C'$.  
By assumption, $G^{(2)}-J^{(2)}$ has a proper 3-coloring.

Note that each vertex $u$ in $J$ satisfies $d_{G^{(2)}}(u) \leq 3$ unless $u$ is a 3-vertex and is also adjacent to a 3-vertex not on $J$.  In that case, $d_{G^{(2)}}(u)=4$; however, then two of the vertices on $\C'$ that are adjacent to $u$ in $G^{(2)}$ have degree 2\ in $G^{(2)}$, and hence, these two vertices can be colored after all of their neighbors.  

To simplify notation, we now speak of finding a proper vertex coloring of $G^{(2)}$.  Let $K = J^{(2)}$.  
We now delete from $K$ all vertices with degree 2 or 4\ in $G^{(2)}$, since they can be colored last; call the resulting graph $\widehat{K}$.  Our goal is to extend the 3-coloring of $G^{(2)}-K$ to $\widehat{K}$; we can then further extend this 3-coloring to $K$.

Observe that each vertex $u$ in $\widehat{K}$ has list size at least 2. Furthermore, all vertices of $\widehat{K}$ have at most two uncolored neighbors in $\widehat{K}$ except for the two vertices $x$ and $y$ that are adjacent on $\C'$ to $v$; however, $d_{\widehat{K}}(x) = d_{\widehat{K}}(y)=3$ and $x$ and $y$ both have three available colors.  
Hence, for every vertex $u$ in $\widehat{K}$, the number of available colors is no smaller than its degree $d_{\widehat{K}}(u)$.  We use this fact as follows.

If any component of $\widehat{K}$ is a path, then we can clearly color each vertex of the component from its available colors, by Theorem~A part (a).  In that case, each component of $\widehat{K}$ contains a vertex $u$ with number of colors greater than $d_{\widehat{K}}(u)$; so by Theorem~A part (a), we can color each vertex of $\widehat{K}$ from its available colors. 
If instead $\widehat{K}$ is a single component, then since the single component is neither a clique nor an odd cycle, we can color all of $\widehat{K}$, by Theorem~B.  

Hence, we only need to consider the case when $\widehat{K}$ has two components: one contains a cycle, where two adjacent vertices have a common neighbor not on the cycle (this is $w$, the off-cycle neighbor of $v$); the other component is a cycle containing $v$.  

Because the first component is neither an odd cycle nor a clique, we can color it by Theorem~B.  
 
Now observe that since $d_H(v)=3$ (and $v$ is not adjacent to a 3-thread on $\C'$), we know that $v$ is adjacent in $G^{(2)}$ to a vertex $z$ not on $C'$ such that $d_{G^{(2)}}(z)=2$.  By uncoloring $z$, we make a third color available at $v$, so we can extend the coloring to the second component of $\widehat{K}$, using Theorem~A part (a).  Lastly, we recolor $z$.
\end{proof}

\section{Planar graphs with low injective chromatic numbers}

In this section, we prove Theorem~\ref{girththm}.  The condition $\Delta\ge 4$ allows us to get better girth conditions than Theorem~\ref{planar-girth}, but not much.   

We prove Theorem~\ref{girththm}(a) first,  and for convenience, we restate the theorem.

\begin{theorem5A}
If $G$ is planar, $\Delta(G)\geq 4$, and $g(G)\geq 9$, then $\chi_i(G) \leq\Delta + 1$.
\end{theorem5A}

\begin{proof}
It is easy to see that $G$ has no 1-vertex and $G$ has no 2-thread.
We also need one more reducible configuration.  Let $u_1,u_2,u_3,u_4,u_5$ be five consecutive vertices along a face $f$.  Suppose that $d(u_1)=d(u_3)=d(u_5)=2$ and $d(u_2)=d(u_4)=3$ and the neighbor of $u_2$ not on $f$ has degree at most 3.  We call this subgraph $H$ and we show that $H$ is a reducible configuration, as follows.  By assumption, $G-u_3$ has an injective coloring with $\Delta+1$ colors; we now modify this coloring to get an injective coloring of $G$.  First uncolor vertices $u_1, u_2, u_4,$ and $u_5$.  Now color the uncolored vertices in the order: $u_2, u_4, u_1, u_5, u_3$.

We use a discharging argument with intial charge $\mu(v) = d(v) - 4$ and $\mu(f) = d(f) - 4$.  We use the following discharging rules.

\begin{itemize}
\item[(R1)] Each face gives charge 1 to each 2-vertex and gives charge 1/3 to each 3-vertex.
\item[(R2)] If face $f$ contains the degree sequence $(4^+,3,2,3,4^+)$, then $f$ gives charge 1/3 to the face adjacent across the 2-vertex. 
\end{itemize}

Now we show that $\mu^*(x)\geq 0$ for each $x\in V(G)\cup F(G)$.
\smallskip

$d(v)=2$: $\mu^*(v)= -2 + 2(1) = 0$.

$d(v)=3$: $\mu^*(v) = -1 + 3(1/3) = 0$.

$d(v)\geq 4$: $\mu^*(v) = \mu(v)\geq 0$.

To argue intuitively without handling (R2) separately, observe that
wherever $(4,3,2,3,4)$ appears we could replace it with $(4,2,4,2,4)$ without creating a 2-thread; after this replacement, face $f$ gives away 1/3 more charge (to account for (R2), so we only consider (R1)).

For each face $f$, let $t_2, t_3, t_4$ denote the number of 2-vertices, 3-vertices, and $4^+$-vertices on $f$, respectively.
The charge of a face $f$ is $\mu^*(f) = d(f) - 4 - t_2 - 1/3t_3 = t_4 + 2/3t_3 - 4$.
If a face has negative charge, then $t_4 + 2/3t_3 < 4$. This implies $3/2t_4 + t_3 < 6$, and hence $t_4+t_3\leq 5$.
Since $G$ contains no 2-threads, $t_2 \leq t_3 + t_4$.
So, if a face has negative charge, $d(f) = t_2 + t_3 + t_4 \leq 2(t_3 + t_4) \leq 10$.  Hence, we only need to verify that $\mu^*(f)\geq 0$ for faces of length at most 10; since $girth(G)\geq 9$, we have only two cases: $d(f)=10$ and $d(f)=9$.

\smallskip

{\bf Case 1: face $f$ of length 10}:
If $\mu^*(f) < 0$, then $t_2 = 5$, $t_3 \geq4$, and $t_4 \leq 1$.
So our degree sequence around $f$ must look like either
(a)~$(2,4^+,2,3,2,3,2,3,2,3)$ or
(b)~$(2,3,2,3,2,3,2,3,2,3)$.

{\bf Case 1a)~$(2,4^+,2,3,2,3,2,3,2,3)$}: 
Let $u_1$ and $u_2$ be the vertices not on $f$ adjacent to the second and third vertices on $f$ of degree 3. 
If $G$ does not contain reducible configuration $H$, then $d(u_1)\geq 4$ and $d(u_2)\geq 4$; but then $f$ receives charge 1/3 by (R2), so $\mu^*(f)\geq 0$.

{\bf Case 1b)~$(2,3,2,3,2,3,2,3,2,3)$}: If any neighbor not on $f$ of a 3-vertex on $f$ has degree at most 3, then $G$ contains reducible configuration $H$.  If all such neighbors have degree at least 4, then $f$ receives charge 1/3 from each adjacent face, so $\mu^*(f)=t_4+\frac23t_3 - 4 + (\frac13)5 = 0 + (\frac23)5 - 4 + (\frac13)5 = 1 > 0$.
\smallskip

{\bf Case 2: face $f$ of length 9}:
If $\mu^*(f) < 0$, then $t_2 = 4$, $t_3 \geq4$, and $t_4 \leq 1$.

Our degree sequence around $f$, beginning and ending with vertices of degree at least 3, must look like one of the four following: (a)~$(3,2,3,2,3,2,3,2,3)$, (b)~$(4^+,2,3,2,3,2,3,2,3)$, (c)~$(3,2,4^+,2,3,2,3,2,3)$, or (d)~$(3,2,3,2,4^+,2,3,2,3)$.

{\bf Case 2a)~$(3,2,3,2,3,2,3,2,3)$}: Because $f$ contains the degree sequence $(2,3,2,3,2,3,2)$,
either $G$ contains the reducible configuration $H$ or $f$ receives
a charge of $\frac13$ from at least two faces.  Hence $\mu^*(f)\geq \frac23(5)-4+\frac13(2)=0$.

{\bf Cases 2b)~$(4^+,2,3,2,3,2,3,2,3)$} and {\bf 2c)~$(3,2,4^+,2,3,2,3,2,3)$}: Again $f$ contains the degree sequence $(2,3,2,3,2)$.  So in each case, if $f$
does not contain the reducible configuration $H$, then $f$ receives a
charge of $\frac13$ from some adjacent face; hence $\mu^*(f)\geq 1+\frac23(4)-4+\frac13(1) = 0$.

{\bf Case 2d)~$(3,2,3,2,4^+,2,3,2,3)$}: 
Let $v_1, w_1, v_2, w_2, v_3, w_3, v_4, w_4, v_5$ denote the vertices on $f$, in order around the face, beginning and ending with 3-vertices. 
If both $v_1$ and $v_2$ are adjacent to vertices of degree at least 4, then $f$ receives a charge of $\frac13$ from an adjacent face, by (R2).  
In this case $\mu^*(f)\geq 1+\frac23(4)-4+\frac13(1) = 0$.
Conversely, we will show that if either $v_1$ or $v_2$ is not adjacent to any vertex of degree at least 4, then $G$ contains a reducible configuration.  

Let $u_1$ and $u_2$ denote the neighbors of $v_1$ and $v_2$ not on $f$.  By minimality, we have an injective coloring of $G-\{w_1,v_2,w_2\}$ with $\Delta+1$ colors.  If $d(u_1) < 4$, then we finish as follows: uncolor $v_1$ and $w_4$, now color $w_1$, $w_2$, $v_2$, $v_1$, $w_4$.
If instead $d(u_2) < 4$, then we finish by coloring $w_1$, $w_2$, $v_2$.
\end{proof}

Now we prove Theorem~\ref{girththm}(b); for convenience, we restate it.

\begin{theorem5B}\label{girth-13}
If $G$ is planar, $\Delta(G)\ge 4$, and $g(G)\geq 13$, then $\chi_i(G) =\Delta$.
\end{theorem5B}

\begin{proof}

Suppose the theorem is false; let $G$ be a minimal counterexample.
Below we note six configurations that must not appear in $G$.
In each case, we can delete the configuration $H$, color $G-H$ (by the minimality of $G$), and extend the coloring to $H$ greedily.

(RC1) $G$ contains no $1$-vertices.

(RC2) $G$ contains no $4$-threads.

(RC3) $G$ contains no $3$-thread with one end having degree $3$.

(RC4) $G$ contains no $2$-threads with both ends having degree $3$.

(RC5) $G$ contains no $3$-vertex that is incident to one $1$-thread and two $2$-threads.

(RC6) $G$ contains no $3$-vertex that is incident to one $2$-thread
and two $1$-threads such that the other end of the one $1$-thread has
degree $3$.

\medskip

For a specified face $f$, let $t_2, t_3$, and $t_4$ be the number of
vertices incident to $f$ of degrees $2$, $3$, and at least $4$. Then
by (RC2) and (RC3), for any face $f$, $t_2\le 2t_3+3t_4$, equality
holds only if every $4^+$-vertex is followed by a $3$-thread and
every $3$-vertex is followed by a $2$-thread. Thus if $t_3>0$, the
equality does not hold. So

\begin{equation}
\label{order-inequality}
t_2\le 2t_3+3t_4;  \mbox{ and if $t_3>0$, then $t_2<2t_3+3t_4$.}
\end{equation}

\medskip

We use a discharging argument. Let the initial charge be
$\mu(x)=d(x)-4$ for $x\in V(G)\cup F(G)$, where $d(x)$ is the degree
of vertex $x$ or the length of face $x$. Then by Euler Formula,
\begin{equation}\label{charge}
 \sum_{x\in V\cup F}\mu(x)=-8.
\end{equation}

We will distribute the charges of the vertices and faces in two
phases. In Phase I, we use a simple discharging rule and show that only three
types of faces have negative charge. In Phase II, we introduce two more discharging
rules and show that the final charge of every face and every vertex is 
nonnegative. We thus get a contradiction to equation (\ref{charge}).
\bigskip

{\bf Discharging Phase I}

We use the following discharging rule.

\begin{enumerate}
\item[(R1)]
Each face $f$ gives charge $\frac13$ to each incident
$3$-vertex and gives charge $1$ to each incident $2$-vertex.
\end{enumerate}

{\bf Remark:} Another way to state this discharging rule is that $f$ gives charge 1 to each vertex, and every $3$-vertex returns charge $\frac23$ and every $4^+$-vertex
returns charge $1$. Thus each vertex has final charge $0$, and the final
charge of each face $f$ is
\begin{equation}\label{final-charge}
\mu^*(f)=\frac23t_3+t_4-4.
 \end{equation}

\begin{clm}\label{bad-faces}
If $\mu^*(f)<0$, then $f$ must have one of the following degree sequences:\\
(a) $(4^+,2,2,2,4^+,2,2,2,4^+,2,2,3,2,2)$\\
(b) $(4^+,2,2,2,4^+,2,2,2,4^+,2,3,2,2)$\\
(c) $(4^+,2,2,2,4^+,2,2,4^+,2,2,3,2,2)$.
\end{clm}

\begin{proof}
Consider a face $f$ with negative charge $\mu^*(f)$. Note that $\mu^*(f)\le
-1/3$; by 
equation (\ref{final-charge}) and inequality (\ref{order-inequality}), we
have $-1/3\ge \mu^*(f)=\frac23t_3+t_4-4\ge \frac13t_2-4$. We rewrite these inequalities as:

\begin{equation}\label{charge-inequality}
t_2\le 2t_3+3t_4\le 11.
\end{equation}
\vspace{-.1in}

By inequality (\ref{charge-inequality}), we know that $t_3+t_4\le 5$. If $t_3+t_4\le 3$,
then $t_2\le 3(t_3+t_4)\le 9$, and hence $d(f)=t_2+(t_3+t_4)\le 9+3 = 12$. Since $girth(G) \geq 13$, this is a contradiction.  So we must have $4\le t_3+t_4\le 5$. 
From inequality (\ref{charge-inequality}), note that $t_4<4$; thus $t_3>0$.

If $t_3+t_4=5$, then inequality (\ref{charge-inequality}) implies that $t_4\le 1$, and thus $t_3\ge 4$. By (RC2),
(RC3), and (RC4), $f$ contains at most two $2$-threads, and three
$1$-threads; thus $d(f)\le 2(2)+3(1)+5=12$. Again, this contradicts $girth(G)\geq 13$, so we must have 
$t_3+t_4=4$.
If $t_2=11$, then $t_2=2t_3+3t_4$. Now inequality (\ref{order-inequality}) implies that $t_3=0$, which is a contradiction. So instead $t_2\le 10$. Combining this inequality with $t_3+t_4=4$, we have $d(f)=t_2+(t_3+t_4)\le 14$.  We now consider two cases: $d(f)=14$ and $d(f)=13$.

If $d(f)=14$, then $t_2=d(f)-(t_3+t_4)=10$. Since $t_3>0$, inequality (\ref{order-inequality})
yields $2t_3+3t_4>10$. So $t_3=1$ and $t_4=3$. By (RC2) and (RC3),
the degree sequence of $f$ must be
$(4^+,2,2,2,4^+,2,2,2,4^+,2,2,3,2,2)$.

If $d(f)=13$, then $t_2=d(f)-(t_3+t_4)=9$. Since $t_3>0$,  inequality (\ref{order-inequality})
yields $2t_3+3t_4>9$. Now $(t_3,t_4)\in \{(2,2),(1,3)\}$.  If
$t_3=t_4=2$, then we have one of the following two cases.
If the two 3-vertices are pseudo-adjacent, then $f$ has at most one 1-thread, two 2-threads, and one 3-thread; so $d(f)\leq 1(1) + 2(2) + 1(3) + 4=12$.
If the two 3-vertices are not pseudo-adjacent, then $f$ has at most four 2-threads, so $d(f)\leq 4(2)+4=12$.  Both of these cases contradict $girth(G)\geq 13$, so we must have $t_4=3$ and $t_3=1$. 
By (RC2) and (RC3), the degree sequence of $f$ must be
$(4^+,2,2,2,4^+,2,2,2,4^+,2,3,2,2)$ or
$(4^+,2,2,2,4^+,2,2,4^+,2,2,3,2,2)$.
\end{proof}

{\bf Discharging Phase II}

Note that each type of bad face listed in Claim~\ref{bad-faces} ends Phase I with charge $-\frac13$.
We now introduce two new discharging rules to send an additional charge of $\frac13$ to these bad faces. 
 
A {\em type-1 $3$-vertex} $u$ is a $3$-vertex that is incident with
one $2$-thread, and two $1$-threads, with the other ends of the
$1$-threads each having degree at least $4$. The vertex $u$ is called a
{\em weak vertex} in the face incident with the two $1$-threads and is called
a {\em strong vertex} in the other two faces incident to $u$.

A {\em type-2 $3$-vertex} $u$ is a $3$-vertex that is incident with
one $0$-thread, one $2$-thread, and one $1^+$-thread, with the other
ends of the $2$-thread and $1^+$-thread having degree at least $4$.
The vertex $u$ is called a {\em slim vertex} in the faces incident
with the $0$-thread and is called a {\em fat vertex} in the other face
incident to $u$.

In our second discharging phase, we use the following two discharging rules.

\begin{enumerate}
\item[(R2)] Each face gives charge $\frac23$ to each of its weak vertices and
gives charge $\frac16$ to each of its slim vertices.

\item[(R3)] Each face receives charge $\frac13$ from each of its strong vertices
and receives charge $\frac13$ from each of its fat vertices.
\end{enumerate}

Let $\mu^{**}(f)$ be the charge after the second discharging phase.  
Let $t_3'(f)$ be the  number of non-special $3$-vertices on $f$, i.e. 3-vertices on $f$ that are not: slim, fat, strong, or weak. 
Beginning with equation (3) and applying rules (R2) and (R3), we write the final charge of a face $f$ as:
\[
\mu^{**}(f)=\frac23t_3'+t_4-4 + 0\cdot \#(weak)+\frac12\cdot \#(slim)+1\cdot \#(strong)+1\cdot \#(fat).
\]

Recall that each vertex had nonnegative charge at the end of Phase I.  Since rules (R2) and (R3) do not change the charge at any vertex, it is clear that every vertex has nonnegative final charge. Now we
will show that every face also has nonnegative final charge. This will contradict equation \ref{charge}.

For a face $f$ with negative charge after Phase I, by
Claim~\ref{bad-faces}, it contains a $3$-vertex $v$ incident to a
$2$-thread and a $1^+$-thread with the other ends having degree at
least $4$. By (RC3) and (RC5), the third thread incident to $v$
is either a $0$-thread, or a $1$-thread. If it is a $0$-thread, then
$v$ is a fat vertex; if it is a $1$-thread, then $v$ is a strong vertex.
In each case, either rule (R2) or (R3) sends an additional charge of $\frac13$ to $f$.  Since $\mu^*(f) = -\frac13$, we conclude that $\mu^{**}(f) = 0$.

If a face has nonnegative charge after Phase I and contains no weak
or slim vertices, then it does not give away charge, and therefore remains
nonnegative. So we only consider the faces containing weak and slim
vertices. Before proceeding, we have the following claims.
Note that even after applying (R2) and (R3), the net charge given from each 3-vertex to each face is nonnegative; we use this fact implicitly when we prove Claims~\ref{4-segment} and~\ref{8-segment}.

\begin{clm}\label{4-segment}
If a face $f$ contains a path $P=u\ldots v$ with $d(u),d(v)\ge 3$,
and if $|V(P)-\{u,v\}|\ge 4$, then $f$ receives a total charge of at least 1 
from the vertices in $V(P)-\{u,v\}$.
\end{clm}

\begin{proof}
Let $P_0=P-\{u,v\}$ and assume $|V(P_0)|\ge 4$. By (RC2), path $P_0$ contains
at least one $3^+$-vertex. If $P_0$ contains a $4^+$-vertex, we are
done. Thus we may assume that $P_0$ contains no $4^+$-vertices.  
Recall that if a $3$-vertex $v$ is weak on face $f$, then each pseudo-neighbor of $v$ that is on $f$ must be a $4^+$-vertex.
Hence, if
$P_0$ contains more than one $3$-vertex, then none of these 3-vertices can be
weak, since each one has a pseudo-neighbor on $f$ that is a $3$-vertex; 
thus $f$ gains at least $\frac23$ from each $3$-vertex, and hence gains more than 1 from $P_0$. So we assume that $P_0$ contains
exactly one $3$-vertex; call it $x$.  By (RC3), vertex $x$ splits $P_0$ into a
$2$-thread and a $1^+$-thread. By (RC5) the third thread incident to $x$
must be a $1$-thread or a $0$-thread. Therefore $x$ is either a fat
vertex or a strong vertex; hence, $f$ receives a total charge of 1 one from $x$.
\end{proof}

To prove Claims~\ref{4-segment},~\ref{clm3}, and~\ref{8-segment} we will be interested in the total charge that $f$ receives from a slim vertex and its pseudo-neighbors and the total charge that $f$ receives from a weak vertex and its pseudo-neighbors.
\begin{clm}
\label{clm3}
A face with negative charge after Phase II must satisfy the
following properties.

(C1) It contains at most one weak vertex.

(C2) It contains at most one slim vertex.

(C3) It does not contain both weak and slim vertices.

\end{clm}
\begin{proof}
Call weak vertices and slim vertices {\it bad} vertices.  Suppose that $f$ has at least 2 bad vertices $v_1$ and $v_2$.  Note that each bad vertex and its pseudo-neighbors give at least 2 to $f$, so if the set of $v_1$ and its pseudo-neighbors is disjoint from the set of $v_2$ and its pseudo-neighbors, then $f$ receives a total of at least $2(2)=4$.  Hence, we must assume these sets are not disjoint; however, in this case $v_1$, $v_2$, and their pseudo-neighbors give $f$ a total of at least three (this case analysis is straightforward).  By Claim~2, the contribution from the remaining vertices is at least 1 (we need to verify that these vertices contain a path of at least four vertices, but this is again straightforward).
\end{proof}

Similar to Claim~\ref{4-segment}, we have the following claim.

\begin{clm}\label{8-segment}
If a face $f$ contains a path $P=u\ldots v$ with $d(u),d(v)\ge 3$ and $|V(P)\setminus\{u,v\}|\ge 8$, then either $P$ contains a slim vertex or else $f$ receives a total charge of at least 2 from the vertices of $P\setminus\{u,v\}$.
\end{clm}

\begin{proof}
Let $P_0=P\setminus\{u,v\}$ and assume $|P_0|\geq 8$.  If $P_0$ contains at least two $4^+$-vertices, then $f$ receives total charge 1 from each of these $4^+$-vertices, and hence receives total charge at least 2 from the vertices of $P_0$.
If $P_0$ contains a single $4^+$-vertex, then call the $4^+$-vertex $y$; note that $P_0$ contains a path $P_1$ with at least 5 vertices, such that one endpoint is $y$ and the other endpoint is adjacent to either $u$ or $v$.  Clearly $f$ receives charge 1 from $y$, and by Claim~\ref{4-segment}, $f$ receives charge at least 1 from the vertices of $P_1-y$; hence, $f$ receives a total charge of 2.

Assume instead that $P_0$ contains no $4^+$-vertices. Note that by (RC2),
path $P_0$ must contain at least two $3$-vertices. If $P_0$ contains at least
three $3$-vertices, then none of them can be weak, since each is pseudo-adjacent to a $3$-vertex.  If also none is slim, then $f$ receives at least $3(\frac23)$ from $P_0$; if $P_0$ contains a slim vertex, then the claim holds.
Finally, if $P_0$ contains at most two $3$-vertices, then by (RC3) and (RC4), it
contains at most two $2$-threads and one $1$-thread; thus
$|V(P_0)|\le 2(2) + 1(1) + 2 =7$, which is a contradiction.
\end{proof}

\begin{clm}
A face with negative charge after phase II must contain no slim vertex and no weak vertex.
\end{clm}

\begin{proof}
Assume that a face $f$ contains exactly one vertex that is weak or slim;
call this vertex $y$. Let the pseudo-neighbors of $y$ on $f$ be $v_1$ and $v_2$.  
There is a path $P$ in $f$ from $v_1$ to $v_2$ such that $y\not\in
P$, and $|V(P)\setminus\{v_1,v_2\}|\ge 8$.  Note that $d(v_1), d(v_2)\ge 3$.
Furthermore, since $y$ is weak or slim, $V(P)\setminus\{v_1,v_2\}$ cannot contain a slim vertex; hence, by Claim~\ref{8-segment}, $f$ receives total charge at least two from $P\setminus\{v_1,v_2\}$. 

If $y$ is weak, then $f$ receives $1+0+1$ from $v_1,y,v_2$; similarly, if $y$ is slim, then $f$ receives at least $1+\frac12+\frac12$ from $v_1,y,v_2$.  In each case, 
we see that $\mu^*(f)\ge -4 + 2 + 2 = 0$; this is a contradiction.
\end{proof}

This completes the proof of Theorem~\ref{girththm}(b).
\end{proof}

\medskip

{\bf Acknowledgement:} The authors would like to thank Professor Alexandr Kostochka for his valuable comments, and Professor Rong Luo for the references on Class 2 graphs.   We also appreciate the valuable comments from referees. 



\end{document}